\documentclass[twoside,11pt,reqno]{amsart}
\usepackage{amsmath,amssymb,amscd,mathrsfs,amscd}
\usepackage{graphics,verbatim}
\usepackage{enumitem}
\usepackage{hyperref}

\newcommand{\losemi}{{\otimes \kern -.78em \ltimes}}
\newcommand{\rosemi}{{\otimes \kern -.78em \rtimes}}

\newcommand{\Hom}{\ensuremath{\operatorname{Hom}}}

\newcommand{\Ind}{\ensuremath{\operatorname{ind}}}

\newcommand{\Ext}{\operatorname{Ext}}

\newcommand{\St}{\operatorname{St}}

\makeatletter
\newcommand{\leqnomode}{\tagsleft@true}
\newcommand{\reqnomode}{\tagsleft@false}
\makeatother

\newtheorem{thm}{Theorem}[section]

\newtheorem{theorem}{Theorem}[subsection]

\makeatletter\let\c@fact\c@theorem\makeatother

\makeatletter\let\c@note\c@theorem\makeatother

\newtheorem{lemma}{Lemma}[subsection]
\makeatletter\let\c@lemma\c@theorem\makeatother

\makeatletter\let\c@lemma\c@theorem\makeatother

\makeatletter\let\c@alg\c@theorem\makeatother

\newtheorem{prop}{Proposition}[subsection]
\makeatletter\let\c@prop\c@theorem\makeatother

\newtheorem{conj}{Conjecture}[subsection]
\makeatletter\let\c@conj\c@theorem\makeatother

\newtheorem{cor}{Corollary}[subsection]
\makeatletter\let\c@cor\c@theorem\makeatother

\makeatletter\let\c@defn\c@theorem\makeatother

\theoremstyle{definition}
\newtheorem{example}{Example}[subsection]

\makeatletter\let\c@remark\c@theorem\makeatother

\makeatletter\let\c@example\c@theorem\makeatother
\numberwithin{equation}{subsection}

\usepackage[capitalise]{cleveref}
\crefname{theorem}{Theorem}{Theorems}
\crefname{fact}{Fact}{Facts}
\crefname{note}{Note}{Notes}
\crefname{lemma}{Lemma}{Lemmas}
\crefname{alg}{Algorithm}{Algorithms}
\crefname{remark}{Remark}{Remarks}
\crefname{example}{Example}{Examples}
\crefname{prop}{Proposition}{Propositions}
\crefname{conj}{Conjecture}{Conjectures}
\crefname{cor}{Corollary}{Corollaries}
\crefname{defn}{Definition}{Definitions}
\crefname{equation}{\!\!}{\!\!} 

\newcounter{listequation}

\begin{document}

\title{The Steinberg quotient of a tilting character}

\author{Paul Sobaje}
\address{Department of Mathematics \\
          Georgia Southern University}
\email{psobaje@georgiasouthern.edu}
\date{\today}

\begin{abstract}
Let $G$ be a simple algebraic group over an algebraically closed field of prime characteristic.  If $M$ is a finite dimensional $G$-module that is projective over the Frobenius kernel of $G$, then its character is divisible by the character of the Steinberg module.  In this paper we study such quotients, showing that if $M$ is an indecomposable tilting module, then the multiplicities of the orbit sums appearing in its ``Steinberg quotient'' are well behaved.
\end{abstract}

\maketitle


\section{Introduction}

\subsection{Motivation}

Let $G$ be a simple and simply connected algebraic group over a field $\Bbbk$ that is algebraically closed of characteristic $p>0$.  Fix $T \le B \le G$, a maximal torus and Borel subgroup of $G$, and denote by $\mathbb{X}$ the set of weights (i.e. the character group of $T$), and by $\mathbb{Z}[\mathbb{X}]$ its group ring.  Let $\mathbb{X}_1 \subseteq \mathbb{X}^+$ denote the $p$-restricted and dominant weights respectively, and $G_1$ is the Frobenius kernel of $G$.

For the complex simple algebraic group $G(\mathbb{C})$, the characters of its simple modules are given by Weyl's character formula.  This formula has a concise presentation as a quotient of elements in $\mathbb{Z}[\mathbb{X}]$, and leads to Weyl's dimension formula, which immediately computes the dimension of any simple $G(\mathbb{C})$-module.  Finding characteristic $p$ versions of these formulas has been a fundamental problem in the study of $G$-modules.

Significant progress has been achieved.  In very large characteristic Lusztig's conjecture \cite{L} gives an answer \cite{AJS} \cite{Fie}.  More recently, following Williamson's \cite{W} counterexamples to Lusztig's conjecture \cite{L} in not-large characteristic, the characters of tilting modules and of projective indecomposable $G_1T$-modules have been brought to the fore.  Thanks to \cite{RW1}, \cite{AMRW}, and \cite{RW2}, it is now known for reasonable bounds on $p$ that the characters of tilting modules and of $G_1T$-PIMs can be given by $p$-Kazhdan-Lusztig polynomials (though the latter are themselves difficult to compute).

However, nothing like Weyl's formula has been discovered in the modular setting.  From the standpoint of character computations, one would like to have a more direct approach.  Of course, the works just mentioned yield much more than mere character data, which is crucial to a complete study of the non-semisimple category of $G$-modules.  Indeed, from \cite{RW1} and \cite{AMRW} one can in principle compute the Weyl filtration factors of any tilting module for $G$ (if $p>h$).  Only some of these are needed for computing the simple characters (cf. \cite{Sob}), but all of them yield information about Ext-groups between Weyl modules.  Consequences of Lusztig's conjecture (when it holds) on the submodule structure of important $G$-modules is detailed in \cite[II.C]{rags}.

\subsection{Stating the problem}

In this paper we study characters of certain tilting modules and $G_1T$-PIMs after first dividing them by the character of the Steinberg module.  Specifically, for each $\lambda \in \mathbb{X}_1$ we define two closely related characters  \footnote{If $p \ge 2h-2$, then $T((p-1)\rho+\lambda) \cong \widehat{Q}_1((p-1)\rho+w_0\lambda)$, so these characters are the same in this case.  It is currently not known exactly when this isomorphism holds, only that it does not hold in general \cite{BNPS2}.}

\begin{align*}
t(\lambda) & = \text{ch}(T((p-1)\rho+\lambda))/\chi((p-1)\rho),\\
\\
q(\lambda) & = \text{ch}(\widehat{Q}_1((p-1)\rho+w_0\lambda))/\chi((p-1)\rho),
\end{align*}

\vspace{0.1in}
\noindent and call these the ``Steinberg quotients" of the respective modules.  They do not correspond to characters of actual $G$-modules nor of $G_1T$-modules (in general), but only of $T$-modules.  These characters are indexed by restricted highest weights, and contain the essential information about these modules.

Let $e(\lambda)$ be the basis element in $\mathbb{Z}[\mathbb{X}]$ corresponding to $\lambda \in \mathbb{X}$.  For $\mu \in \mathbb{X}^+$, let $W\mu$ be the $W$-orbit of $\mu$, and set
$$s(\mu)=\sum_{\nu \in W\mu} e(\nu).$$

\vspace{0.1in}
Both $t(\lambda)$ and $q(\lambda)$ are elements in $\mathbb{Z}[\mathbb{X}]^W$, therefore can be expressed as linear combinations of orbit sums (cf. Section \ref{S:quotients}).  Thus, there are integers $a_{\mu}^{\lambda}, b_{\mu}^{\lambda}$ such that
$$q(\lambda) = \sum_{\mu \in \mathbb{X}^+} a_{\mu}^{\lambda} s(\mu), \qquad t(\lambda)  = \sum_{\mu \in \mathbb{X}^+} b_{\mu}^{\lambda} s(\mu).$$
The most immediate problem is to determine these coefficients.   This is equivalent to the problem of computing the baby Verma multiplicities of $T((p-1)\rho+\lambda)$ and of $\widehat{Q}_1((p-1)\rho+w_0\lambda)$.  In this paper we show that the coefficients $b_{\mu}^{\lambda}$ follow a well-behaved pattern, and we conjecture that the $a_{\mu}^{\lambda}$ do as well.  We further expect that both will be useful in deriving more compact character formulas for simple $G$-modules.

\subsection{Main results}

Basic information on these coefficients can be obtained.  Since $\widehat{Q}_1((p-1)\rho+w_0\lambda)$ is a $G_1T$-summand of $T((p-1)\rho+\lambda)$, it follows that
$$a_{\mu}^{\lambda} \le b_{\mu}^{\lambda}.$$
Also, since the highest weight $(p-1)\rho+\lambda$ occurs once in each module,
$$a_{\lambda}^{\lambda} = b_{\lambda}^{\lambda} = 1.$$ 
One can also show that if either $a_{\mu}^{\lambda} > 0$, or $b_{\mu}^{\lambda} > 0$, then it is necessary that
$$\mu-\rho \uparrow \lambda-\rho,$$
(see Section 2.3 for this notation) or equivalently that
$$\mu+(p-1)\rho \uparrow \lambda+(p-1)\rho.$$

Our first main result is to show that this is both a necessary and sufficient condition for an orbit sum to appear in $b_{\mu}^{\lambda}$.

\begin{theorem}\label{T:nondecreasing}
For $\mu \in \mathbb{X}^+$, $b_{\mu}^{\lambda} > 0$ if and only if
$$\mu-\rho \uparrow \lambda-\rho.$$
Further, if $\mu^{\prime} \in \mathbb{X}^+$ and
$$\mu-\rho \; \uparrow \; \mu^{\prime}-\rho \; \uparrow \; \lambda-\rho,$$
then $b_{\mu}^{\lambda} \ge b_{\mu^{\prime}}^{\lambda}$.
\end{theorem}

\begin{conj}
The same inequalities also hold for the coefficients $a_{\mu}^{\lambda}$.
\end{conj}

This result says that in $t(\lambda)$, every orbit sum that can appear does appear.  Moreover, the multiplicity of occurrence is nondecreasing as the orbits get smaller (if $\mu$ and $\mu^{\prime}$ are dominant and $\mu-\rho \; \uparrow \; \mu^{\prime}-\rho$, then $W\mu$ is in the interior of the convex hull of $W\mu^{\prime}$).

Theorem \ref{T:nondecreasing} was essentially first proved by Ye \cite{Ye}, under the condition that $p \ge 2h-2$ and that $\lambda$ is a $p$-regular weight.  Doty and Sullivan \cite{DS} then extended this to all restricted $\lambda$ (under the same bound on $p$).  Our proof draws on many of the key arguments from these papers.

We can get additional information about these characters by relating them to Hom-spaces over $G_1$.  We introduce some additional notation.  We can define the $\mathbb{Z}$-linear map
$$\pi_p:\mathbb{Z}[\mathbb{X}] \rightarrow \mathbb{Z}[\mathbb{X}]$$
by setting $\pi_p(e(\mu))=e(\mu/p)$ if $\mu \in p\mathbb{X}$, and $0$ otherwise.

\begin{thm}\label{T:hombychar}
For all $\lambda,\mu \in \mathbb{X}_1$, there are equalities of characters
$$\pi_p(q(\lambda)q(\mu)) = \textup{ch}\left(\Hom_{G_1}(\widehat{Q}_1((p-1)\rho-\lambda),\widehat{Q}_1((p-1)\rho+w_0\mu))^{(-1)}\right),$$
and 
$$\pi_p(t(\lambda)t(\mu))=\textup{ch}\left(\Hom_{G_1}(T((p-1)\rho+w_0\lambda),T((p-1)\rho+\mu))^{(-1)}\right).$$
Additionally, for all $\lambda,\mu \in \mathbb{X}_1$, $\pi_p(q(\lambda)q(\mu))$ is the character of some $G$-module, and $\pi_p(t(\lambda)t(\mu))$ is the character of a tilting module.
\end{thm}

\subsection{Relation to existing literature}

The characters $q(\lambda)$ appeared also in \cite[10.3]{Hu}, where they are denoted as $q(1,\lambda)$, and in \cite{D2}, where they are denoted as $\dot{\psi}_{\lambda}$ (in both cases the indexing on highest weights differs from ours).  In the first case, reference is made to even earlier literature in which such characters are used in arguments comparing the formal characters of the $G_1T$-PIMs with the Brauer characters of the PIMs for finite Chevalley groups.  Orbit multiplicities for the $q(\lambda)$ do not appear to be in view in either setting,  though one might be able to obtain some of this information from \cite[Theorem 10.11]{Hu}.  The discussion at the beginning of \cite[10.12]{Hu} also shows that the idea of using ``Steinberg quotients'' in the study of formal characters of $G$-modules is an old idea.

The statements found in this paper are inspired by the work of Ye \cite{Ye} on the composition multiplicities of baby Verma modules.  By Brauer-Humphreys reciprocity, that problem is equivalent to determining the orbit multiplicities in $q(\lambda)$.  However, Ye only proves his result for regular weights under the condition that $p \ge 2h-2$.  Though his paper contains the main ideas that appear here, it is not clear from the exposition that they can be generalized to smaller characteristics.  Indeed, it is only in Doty and Sullivan's work \cite{DS} that it is shown that Ye's statements also hold for singular weights under the bound $p\ge 2h-2$.

A standard question to ask whenever a result about $G_1T$-PIMs is stated under the bound $p \ge 2h-2$ is whether or not this bound is necessary, or if in fact what one really needs is the lifting of the $G_1T$-PIMs to tilting modules for $G$.  That is, even if the PIMs do lift when $p<2h-2$, certain other properties break down, including the characterization of the lifts as projective objects in a truncated subcategory of $G$-modules, and it may happen that one of these additional properties is also needed for the result to stand.  Our main result answers this question by proving the stronger statement that in fact the orbit multiplicity pattern holds for the tilting modules $T((p-1)\rho+\lambda)$, with $\lambda$ restricted, regardless of whether or not these are modules are indecomposable over $G_1T$.

Finally, the baby Verma multiplicities of $T((p-1)\rho+\lambda)$ can be found, when $p>h$, using the results in \cite{AMRW}, and are directly given in \cite{RW2} when $p \ge 2h-2$.  In both cases, the answer comes in terms of evaluations of $p$-Kazhdan-Lusztig polynomials, and the authors of these papers note that the process of obtaining the relevant polynomials is difficult to carry out in practice.  In comparison, our results on the one hand do not lead to precise multiplicities, but on the other hand give an easily stated pattern that holds for all $p$.

\subsection{Acknowledgements} This paper was in a preliminary stage when I began a two-week visit with Stephen Donkin at the University of York.  Time spent working with Steve has greatly impacted the present form of this paper, and I am grateful for Steve's generosity, for communicating shorter proofs of some of these results, and for directing me to the paper by Ye.  I thank the London Mathematical Society, and the University of York, for financial support during this visit.

I would also like to thank Henning Haahr Andersen for sending a number of detailed comments and thought-provoking questions after reading an earlier draft of this paper.

\section{Recollections}

We recall basic results from the literature.  All notation not detailed here (or in the introduction) will follow that in \cite{rags}.

\subsection{Roots and weights}
We fix a maximal torus $T$ inside a Borel subgroup $B$ of $G$.  The root system is denoted $\Phi$, and we fix a set of simple roots $S=\{\alpha_1,\alpha_2,\ldots,\alpha_n\}$, where $n$ is the rank of $T$.  This determines a set of positive roots $\Phi^+ \subset \Phi$.  Denote by $\mathbb{X}$ the character group of $T$ (also called the set of weights).  For each $1 \le i \le n$ there is a fundamental dominant weight $\varpi_i$ determined by the property that $\langle \varpi_i, \alpha_j^{\vee} \rangle = \delta_{ij}$.  The subsets $\mathbb{X}_1 \subseteq \mathbb{X}^+ \subseteq \mathbb{X}$ denote the $p$-restricted dominant weights and dominant weights respectively.  The element $\rho$ is the half-sum of the positive roots, and $W$ is the Weyl group.  The highest short root is $\alpha_0$, so that $\alpha_0^{\vee}$ is the highest coroot, and $w_0$ is the longest element of $W$.

\subsection{$G$-modules and $G_1T$-modules}
For each $\lambda \in \mathbb{X}^+$ there is a simple $G$-module $L(\lambda)$, a costandard module $\nabla(\lambda)=\Ind_B^G \lambda$, a standard module $\Delta(\lambda)=(\Ind_B^G -w_0\lambda)^*$, and an indecomposable tilting module $T(\lambda)$.

For each $\lambda \in \mathbb{X}$ there is a simple $G_1T$-module $\widehat{L}_1(\lambda)$, a projective indecomposable $G_1T$-module $\widehat{Q}_1(\lambda)$, and ``baby Verma modules"
$$\widehat{Z}_1(\mu) = \textup{coind}_{B_1^+T}^{G_1T} \mu, \qquad \widehat{Z}_1^{\prime}(\mu) = \textup{ind}_{B_1T}^{G_1T} \mu.$$

Fix a Frobenius endomorphism $F:G \rightarrow G$.  For any $G$-module $M$, we denote by $M^{(1)}$ its twist under $F$.

\subsection{The Affine Weyl Group $W_p$}

For each $\alpha \in \Phi^+$ and $n \in \mathbb{Z}$ there is an affine reflection on $\mathbb{X}$ defined by
$$s_{\alpha,n}(\lambda)=s_{\alpha}(\lambda) + n\alpha.$$
The affine Weyl group $W_p$ is the group generated by all affine reflections of the form $s_{\alpha,np}$ for all $\alpha \in \Phi$ and all $n \in \mathbb{Z}$.  It is isomorphic to the semidirect product $W \ltimes p\mathbb{Z}\Phi$.

The dot action of $W_p$ on $\mathbb{X}$ is given by
$$w \bullet \lambda = w(\lambda+\rho)-\rho.$$
There is an equivalence relation $\uparrow$ defined by declaring that $\lambda \uparrow \mu$ if there is a sequence of affine reflections (of the form just stated) $s_1,s_2,\ldots,s_m$ such that
\begin{equation}\label{E:alcovestring}
\lambda \le s_1 \bullet \lambda \le s_2 \bullet s_1 \bullet \lambda \le \cdots \le s_m \bullet \cdots \bullet s_1 \lambda = \mu.
\end{equation}
We will need the following important observation of Andersen.

\begin{prop}\label{P:And}\cite[Proposition 1]{And}
If $\lambda,\mu \in \mathbb{X}^+-\rho$, and $\lambda \uparrow \mu$, then we may choose affine reflections so that every weight in (\ref{E:alcovestring}) is in $\mathbb{X}^+-\rho$.
\end{prop}

\subsection{Linkage and Blocks}

Given $\lambda \in \mathbb{X}^+$, the block in $G$-Mod that contains $\lambda$ is contained in the set
$$(W_p \bullet \lambda) \cap \mathbb{X}^+.$$
If $\lambda \in \mathbb{X}$, then the block in $G_1T$-Mod that contains $\lambda$ is contained in
$$W_p \bullet \lambda.$$

\begin{lemma}\label{L:linkedtohighest}
Let $M$ be a $G$-module having $L(\lambda)$ as a composition factor.   Suppose that for every composition factor $L(\mu)$ of $M$, $\mu \uparrow \lambda$.  If $\widehat{L}_1(\gamma)$ is a $G_1T$-composition factor of $M$, then
$$\gamma \uparrow \lambda.$$
\end{lemma}

\begin{proof}
Write $\gamma=\gamma_0 + p\gamma_1$ with $\gamma_0 \in \mathbb{X}_1$.  A composition series over $G_1T$ may be given by refining a composition series over $G$.  Thus, there is some weight $\sigma \in \mathbb{X}^+$ such that $\widehat{L}_1(\gamma)$ is a $G_1T$-composition factor of $L(\gamma_0) \otimes L(\sigma)^{(1)}$, which is in turn a $G$-composition factor of $M$.  By assumption,
$$\gamma_0 + p\sigma \uparrow \lambda.$$
Since $p\sigma$ is the highest weight of $L(\sigma)^{(1)}$, we have that $p\sigma=p\gamma_1 + p\beta$ for some $\beta \in \mathbb{Z}_{\ge 0}\Phi^+$.  Thus
$$\gamma \uparrow \gamma + p\beta = \gamma_0 + p\sigma \uparrow \lambda.$$
\end{proof}

\subsection{Characters of $T$-modules}

The Grothendieck ring of the category of finite dimensional $T$-modules is isomorphic to the group algebra $\mathbb{Z}[\mathbb{X}]$.  For each $\mu \in \mathbb{X}$, we denote by $e(\mu)$ the corresponding basis element in $\mathbb{Z}[\mathbb{X}]$.  Since $\mathbb{Z}[\mathbb{X}]$ is the group algebra of a free abelian group of rank $n$, it is isomorphic to the ring of Laurent polynomials over $\mathbb{Z}$ in $n$ indeterminants.  In particular, $\mathbb{Z}[\mathbb{X}]$ is an integral domain, so that the cancellation property holds.

We denote by $s(\mu)$ the sum of the weights in the $W$-orbit of $\mu$.  These elements form a basis of $\mathbb{Z}[\mathbb{X}]^W$.  The characters of finite dimensional $G_1T$-modules do not in general belong to $\mathbb{Z}[\mathbb{X}]^W$.  However, if $\lambda \in \mathbb{X}_1$, then Donkin proved that $\text{ch}(\widehat{Q}_1(\lambda))$ is always equal to the character of some $G$-module \cite{D1}.

For any character $\zeta = \sum a_{\mu}e(\mu)$, the ``dual" of $\zeta$ is
$$\zeta^*=\sum a_{\mu}e(-\mu),$$
and the ``Frobenius twist" of $\zeta$ is
$$\zeta^F=\sum a_{\mu}e(p\mu).$$
If $\zeta = \text{ch}(M)$ for a $T$-module $M$, then $\zeta^* = \text{ch}(M^*)$, and $\zeta^F = \text{ch}(M^{(1)})$.

\subsection{Euler characteristic}
We recall several facts that can be found in detail in \cite[II.5]{rags}.  Given a finite-dimensional $B$-module $M$, its Euler characteristic is 
$$\chi(M)=\sum_{i \ge 0} (-1)^i \textup{ch}(R^i \Ind_B^G \, M).$$
If $\mu \in \mathbb{X}^+$, then Kempf's vanishing theorem implies that
$$\textup{ch}(\nabla(\mu)) = \chi(\mu).$$
For $\mu \in \mathbb{X}$, define
$$A(\mu)=\sum_{w \in W} -1^{\ell(w)}e(w\mu).$$
We note that this sum runs over $W$, unlike $s(\mu)$, which sums only over the elements in the orbit $W\mu$.

\begin{theorem}\label{T:EulerFacts}
The Euler characteristic has the following properties:
\begin{enumerate}
\item For every $\mu \in \mathbb{X}$ and $w \in W$, 
$$\chi(w\bullet \mu)=(-1)^{\ell(w)}\chi(\mu).$$
\item For every $\lambda,\mu \in \mathbb{X}$,
$$\chi(\lambda)s(\mu)=\sum_{\gamma \in W\mu} \chi(\lambda+\gamma).$$
\item For every $\mu \in \mathbb{X}$,
$$\chi((p-1)\rho+p\mu)=\chi((p-1)\rho)\cdot \chi(\mu)^F$$
\item For every $\mu \in \mathbb{X}$,
$$\chi(\mu)=A(\mu+\rho)/A(\rho).$$
\end{enumerate}
\end{theorem}

Property (2) is referred to as Brauer's formula, and plays a large role in the next section.

\section{Steinberg Quotients}

\subsection{}\label{S:quotients}

The module $\widehat{Z}_1(\mu)$ is projective and injective over $B_1T$, while $\widehat{Z}_1^{\prime}(\mu)$ is projective and injective over $B_1^+T$.  By \cite[Proposition II.11.2]{rags}, any $G_1T$-module that is projective over $B_1T$ (resp. $B_1^+T$) has a filtration by submodules whose successive quotients are of the form $\widehat{Z}_1(\mu)$, $\mu \in \mathbb{X}$ (resp. $\widehat{Z}_1^{\prime}(\mu)$, $\mu \in \mathbb{X}$).

The characters of $\widehat{Z}_1(\mu)$ and $\widehat{Z}_1^{\prime}(\mu)$ are the same, and are a translate of the Steinberg character.  Specifically,
$$\text{ch}(\widehat{Z}_1(\mu)) = \text{ch}(\widehat{Z}_1^{\prime}(\mu)) = \chi((p-1)\rho)e(\mu-(p-1)\rho).$$
If $M$ is a finite dimensional projective $G_1T$-module, then $M$ has filtrations of the forms just described.  This shows that
\begin{equation}\label{E:quotientWstable}
\text{ch}(M)/\chi((p-1)\rho) = \sum_{\mu \in \mathbb{X}} c_{\mu} e(\mu),  \quad c_{\mu}=\textup{mult. of $\widehat{Z}_1(\mu+(p-1)\rho)$ in $M$}.
\end{equation}
One can see for $M$ as above that
$$\text{ch}(M) \in \mathbb{Z}[\mathbb{X}]^W \Rightarrow \text{ch}(M)/\chi((p-1)\rho) \in \mathbb{Z}[\mathbb{X}]^W.$$
Indeed, multiplying each side in (\ref{E:quotientWstable}) by $\chi((p-1)\rho)$, we get
$$\text{ch}(M) = \chi((p-1)\rho) \sum_{\mu \in \mathbb{X}} c_{\mu} e(\mu).$$
The left hand side is $W$-invariant, so the right hand side is also.  It follows from the $W$-invariance of $\chi((p-1)\rho)$ that for each $w \in W$, we get
$$\chi((p-1)\rho) \sum_{\mu \in \mathbb{X}} c_{\mu} e(\mu) = \chi((p-1)\rho) \sum_{\mu \in \mathbb{X}} c_{\mu} e(w\mu).$$
Using the cancellation property of the integral domain $\mathbb{Z}[\mathbb{X}]$, it now follows that
$$\sum_{\mu \in \mathbb{X}} c_{\mu} e(\mu) = \sum_{\mu \in \mathbb{X}} c_{\mu} e(w\mu),$$
so that $c_{\mu}=c_{w\mu}$.  This holds for all $\mu$ and $w$.  In conclusion, if $M$ is projective over $G_1T$ and $\text{ch}(M) \in \mathbb{Z}[\mathbb{X}]^W$, then 
$$\text{ch}(M)/\chi((p-1)\rho) = \sum_{\mu \in \mathbb{X}^+} c_{\mu} s(\mu), \quad c_{\mu} \ge 0.$$
Tilting modules are $G$-modules, and for $\lambda \in \mathbb{X}_1$, Donkin has shown that $\text{ch}(\widehat{Q}_1(\lambda))$ is the character of some $G$-module.  Therefore both $t(\lambda)$ and $q(\lambda)$ are in $\mathbb{Z}[\mathbb{X}]^W$, and it follows from above that there exist coefficients $a_{\mu}^{\lambda}, b_{\mu}^{\lambda} \in \mathbb{Z}_{\ge 0}$ such that
$$q(\lambda) = \sum_{\mu \in \mathbb{X}^+} a_{\mu}^{\lambda} s(\mu), \qquad t(\lambda)  = \sum_{\mu \in \mathbb{X}^+} b_{\mu}^{\lambda} s(\mu).$$

\subsection{}

In this subsection we will prove Theorem \ref{T:nondecreasing}.  We begin with a lemma that easily follows from above.

\begin{lemma}
For every $\lambda \in \mathbb{X}_1$ and $\sigma \in \mathbb{X}^+$,
$$\textup{ch}(T((p-1)\rho+\lambda) \otimes \nabla(\sigma)^{(1)}) = \sum_{\mu \in \mathbb{X}^+} \sum_{\gamma \in W\mu} b_{\mu}^{\lambda} \cdot \chi((p-1)\rho+p\sigma+\gamma).$$
\end{lemma}

\begin{proof}
We have
\begin{align*}
\textup{ch}(T((p-1)\rho+\lambda) \otimes \nabla(\sigma)^{(1)}) & = (\chi((p-1)\rho)\sum_{\mu \in \mathbb{X}^+} b_{\mu}^{\lambda} s(\mu))\cdot \chi(\sigma)^F\\
& = \left(\chi((p-1)\rho)\chi(\sigma)^F\right) \sum_{\mu \in \mathbb{X}^+} b_{\mu}^{\lambda} s(\mu)\\
& = (\chi((p-1)\rho+p\sigma) \sum_{\mu \in \mathbb{X}^+} b_{\mu}^{\lambda} s(\mu)\\
& = \sum_{\mu \in \mathbb{X}^+} \sum_{\gamma \in W\mu} b_{\mu}^{\lambda} \cdot \chi((p-1)\rho+p\sigma+\gamma),\\
\end{align*}
where the last two equalities come from Theorem \ref{T:EulerFacts}(3) and \ref{T:EulerFacts}(2) respectively.
\end{proof}

We can now give the main idea in the proof of the next theorem.  First, the module $T((p-1)\rho+\lambda) \otimes \nabla(\sigma)^{(1)}$ has a good filtration, because it is a direct summand of the good filtration module
$$\St \otimes \nabla(\sigma)^{(1)} \otimes T(\lambda) \cong \nabla((p-1)\rho+p\sigma) \otimes T(\lambda).$$
Thus, the sum
$$\sum_{\mu \in \mathbb{X}^+} \sum_{\gamma \in W\mu} b_{\mu}^{\lambda} \cdot \chi((p-1)\rho+p\sigma+\gamma),$$
has an equivalent expression in terms of Weyl characters (i.e. Euler characteristics of dominant weights), and in that expression the coefficients must be non-negative.  If for some $\mu$ and some $\gamma \in W\mu$ the weight
$$(p-1)\rho+p\sigma+\gamma$$
is not dominant, then we make it dominant by the identity
$$\chi((p-1)\rho+p\sigma+\gamma)=(-1)^{\ell(w)}\chi(w\bullet((p-1)\rho+p\sigma+\gamma)).$$
(Strictly speaking, $(p-1)\rho+p\sigma+\gamma$ also must not lie on a reflecting hyperplane with respect to the dot action.)  The proof will then follow by choosing, for
$$\mu-\rho \uparrow \mu^{\prime}-\rho \uparrow \lambda-\rho$$
an appropriate $\sigma$ such that one of the coefficients in the expansion of Weyl characters is $b_{\mu}^{\lambda} - b_{\mu^{\prime}}^{\lambda}$.  This idea, and its execution, is an adaptation of the approach used in \cite{Ye}.

\begin{theorem}
Write $t(\lambda) = \sum_{\mu} b_{\mu}^{\lambda} s(\mu)$.  Then $b_{\mu}^{\lambda} > 0$ if and only if
$$\mu-\rho \uparrow \lambda-\rho.$$
Further, if
$$\mu-\rho \uparrow \mu^{\prime}-\rho \uparrow \lambda-\rho,$$
then $b_{\mu}^{\lambda} \ge b_{\mu^{\prime}}^{\lambda}$.
\end{theorem}

\begin{proof}
Suppose that $\mu$ is dominant and that 
$$\mu-\rho \uparrow \lambda-\rho.$$
It follows from Proposition \ref{P:And} that there is some dominant weight $\mu^{\prime} \le \lambda$, and some positive root $\alpha$ and $n$ such that
$$\mu-\rho \le s_{\alpha,np} \bullet (\mu-\rho) = \mu^{\prime}-\rho.$$
Since $b_{\lambda}^{\lambda}=1$, both statements in the theorem will follow once it is established that $b_{\mu}^{\lambda} \ge b_{\mu^{\prime}}^{\lambda}$.

We may assume that $\mu < \mu^{\prime}$.  Since $s_{\alpha,np}$ is an affine reflection,
$$s_{\alpha,np} \bullet (\mu^{\prime}-\rho) = \mu-\rho.$$
Let $m$ be such that
$$\langle \mu^{\prime}, \alpha^{\vee} \rangle = m.$$
Since $\mu^{\prime} \ge \mu$ it follows that $m \ge np$, and since both $\mu^{\prime}$ and $\mu$ are dominant, it further follows that $n>0$.
We have that
$$\mu^{\prime} - (m-np)\alpha = \mu.$$
Choose $w \in W$ and $\alpha_i$ a simple root such that $w\alpha=-\alpha_i$.  Then
$$\langle w\mu^{\prime}, \alpha_i^{\vee} \rangle = -m$$
and
$$s_{\alpha_i,-np} \bullet (w\mu^{\prime}-\rho) = w\mu^{\prime}-\rho-(np-m)\alpha_i =w\mu-\rho.$$
We choose
$$\sigma = \sum c_j \varpi_j \in \mathbb{X}^+$$
with $c_i = n-1$, and $c_j >>0$ for $i \ne j$.  We note that since $n>0$, $c_i \ge 0$, so $\sigma \in \mathbb{X}^+$.

We have
\begin{align*}
s_{\alpha_i} \bullet (w\mu^{\prime}+(p-1)\rho + p\sigma) & = s_{\alpha_i}(w\mu^{\prime}+p\rho + p\sigma)-\rho\\
& = s_{\alpha_i}(w\mu^{\prime}+p(\rho + \sigma))-\rho\\
& = w\mu^{\prime}+p(\rho + \sigma)-\rho - (-m+np)\alpha_i\\
& = w\mu+(p-1)\rho + p\sigma.\\
\end{align*}
Therefore
$$\chi(w\mu^{\prime}+(p-1)\rho + p\sigma) = -\chi(w\mu+(p-1)\rho + p\sigma),$$
and that
$$w\mu+(p-1)\rho + p\sigma \in \mathbb{X}^+.$$
Further, for any $y \in W$ with $y \ne s_{\alpha_i}$, we can guarantee that
$$y \bullet (w\mu+(p-1)\rho + p\sigma) - ((p-1)\rho + p\sigma)$$
is not a weight of $t(\lambda)$, by choosing the $c_j$ to be sufficiently large.  Specifically, we have that
$$y \bullet (w\mu+(p-1)\rho + p\sigma) - ((p-1)\rho + p\sigma) = yw\mu + p[(y\rho-\rho) +(y\sigma-\sigma)].$$
The largest and smallest weights of $t(\lambda)$ are $\lambda$ and $w_0\lambda$ respectively.  Since $w_0\lambda \le yw\mu \le \lambda$,
we choose $c_j$ to be large enough that
$$\text{ht}(p[(y\rho-\rho) +(y\sigma-\sigma)]) < \text{ht}(w_0\lambda-\lambda),$$
or equivalently that
$$\text{ht}(p(y\sigma-\sigma)) < \text{ht}(w_0\lambda-\lambda)-\text{ht}(p(y\rho-\rho)),$$
where $\text{ht}$ is the height function on elements in $\mathbb{Z}\Phi$ (cf. \cite[II.4.8(5)]{rags}).  It follows that $yw\mu + p[(y\rho-\rho) +(y\sigma-\sigma)]$ is not a weight of $t(\lambda)$.

From this we conclude that when expressing the character
$$\textup{ch}(T((p-1)\rho+\lambda) \otimes \nabla(\sigma)^{(1)})$$
in the basis of Weyl characters, the coefficient of
$$\chi(w\mu+(p-1)\rho + p\sigma)$$
is
$$b_{\mu}^{\lambda}-b_{\mu^{\prime}}^{\lambda}.$$
Since $T((p-1)\rho+\lambda) \otimes \nabla(\sigma)^{(1)}$ has a good filtration, this coefficient cannot be negative, therefore
$$b_{\mu}^{\lambda} \ge b_{\mu^{\prime}}^{\lambda}.$$
Conversely, if $b_{\mu}^{\lambda} \ne 0$, then $\widehat{Z}_1((p-1)\rho+\mu)$ is a baby Verma factor of $T((p-1)\rho+\lambda)$, so $\widehat{L}_1((p-1)\rho+\mu)$ is a $G_1T$-composition factor of $T((p-1)\rho+\lambda)$.  Every $G$-composition factor $L(\gamma)$ of $T((p-1)\rho+\lambda)$ satisfies $\gamma \uparrow (p-1)\rho+\lambda$ (cf. \cite[E.3(2)]{rags}).  Applying Lemma \ref{L:linkedtohighest},
$$(p-1)\rho+\mu \uparrow (p-1)\rho+\lambda,$$
which is equivalent to
$$\mu-\rho \uparrow \lambda-\rho.$$
\end{proof}

\subsection{}

Recall that for any finite dimensional $T$-module $V$, the character of $T$ is said to be proper if it is equal to the character of some rational $G$-module.  We thank Stephen Donkin for communicating the following lemma.

\begin{lemma}
Let $\lambda \in X_1$.  If $V$ is a finite-dimensional $G_1T$-module such that $\textup{ch}(V)$ is proper, then
$$\textup{ch}(\Hom_{G_1}(\widehat{Q}_1(\lambda),V))$$
is also proper.
\end{lemma}

\begin{proof}
Since $V$ has a proper character, there are weights $\lambda_1,\lambda_2,\ldots,\lambda_m \in \mathbb{X}_1$ and $\mu_1,\mu_2,\ldots,\mu_m \in \mathbb{X}^+$ such that
$$\textup{ch}(V)=\sum_{i=1}^m \left[\textup{ch}(L(\lambda_i))\textup{ch}(L(\mu_i)^{(1)})\right].$$
Since $\widehat{Q}_1(\lambda)$ is projective over $G_1$, the functor $\Hom_{G_1}(\widehat{Q}_1(\lambda),\underline{\quad})$ is exact.  From this it follows that, as $T$-modules, 
$$\Hom_{G_1}(\widehat{Q}_1(\lambda),V) \cong \bigoplus_{i=1}^m \left[\Hom_{G_1}(\widehat{Q}_1(\lambda),L(\lambda_i))\otimes L(\mu_i)^{(1)}\right].$$
But $\Hom_{G_1}(\widehat{Q}_1(\lambda),L(\lambda_i))$ is only nonzero if $\lambda_i=\lambda$ (we are assuming that $\lambda$ and $\lambda_i$ are both restricted), in which case it is isomorphic to $\Bbbk$, therefore there is a subcollection of indices such that
$$\textup{ch}\left(\Hom_{G_1}(\widehat{Q}_1(\lambda),V)\right) = \sum_{j=1}^k \textup{ch}(L(\mu_{i_j})^{(1)}).$$
The character on the right hand side is clearly proper, so the left hand side is also.
\end{proof}

\begin{cor}
The $T$-module $\Hom_{G_1}(\widehat{Q}_1(\lambda),\widehat{Q}_1(\mu))$ has a proper character for all $\lambda,\mu \in \mathbb{X}_1$.
\end{cor}

\begin{proof}
This follows from the previous result together with the main result of \cite{D1}, where it was proved that $\widehat{Q}_1(\mu)$ has a proper character if $\mu \in \mathbb{X}_1$.
\end{proof}

\begin{lemma}
The $G$-module
$$\Hom_{G_1}(T((p-1)\rho+\lambda),T((p-1)\rho+\mu))^{(-1)}$$
is tilting for all $\lambda,\mu \in \mathbb{X}^+$.
\end{lemma}

\begin{proof}
This module is a $G$-summand of the $G$-module
$$\Hom_{G_1}(\St \otimes T(\lambda),\St \otimes T(\mu))^{(-1)},$$
which is isomorphic as a $G$-module to
$$\Hom_{G_1}(\St,\St \otimes T(\mu) \otimes T(\lambda)^*)^{(-1)}.$$
The result now follows from the fact that $\St \otimes T(\mu) \otimes T(\lambda)^*$ is a tilting module for $G$, together with the fact that the functor
$$\Hom_{G_1}(\St,?)^{(-1)}$$
defines an equivalence of categories from the $G_1$-Steinberg block of $G$-Mod to $G$-Mod \cite[II.10.4]{rags}.
\end{proof}

We can now prove Theorem \ref{T:hombychar}, which we will restate for the convenience of the reader.

\begin{thm}
For all $\lambda,\mu \in \mathbb{X}_1$, there are equalities of characters
$$\pi_p(q(\lambda)q(\mu)) = \textup{ch}\left(\Hom_{G_1}(\widehat{Q}_1((p-1)\rho-\lambda),\widehat{Q}_1((p-1)\rho+w_0\mu))^{(-1)}\right),$$
and 
$$\pi_p(t(\lambda)t(\mu))=\textup{ch}\left(\Hom_{G_1}(T((p-1)\rho+w_0\lambda),T((p-1)\rho+\mu))^{(-1)}\right).$$
\end{thm}

\begin{proof}
We prove the first case, the second statement is similar.  The module $\widehat{Q}_1((p-1)\rho-\lambda)$ has highest weight
$$2(p-1)\rho+w_0((p-1)\rho-\lambda) = (p-1)\rho-w_0\lambda,$$
therefore has a $G_1T$-filtration by modules of the form $\widehat{Z}_1((p-1)\rho+\gamma)$, where $\gamma$ ranges all the weights of $q(-w_0\lambda)$.  Similarly, $\widehat{Q}_1((p-1)\rho+w_0\mu)$ has a filtration by modules of the form $\widehat{Z}^{\prime}_1((p-1)\rho+\sigma)$ as $\sigma$ ranges over the weights of $q(\mu)$.

For all $\gamma$ and $\sigma$, 
$$\Ext^1_{G_1}(\widehat{Z}_1((p-1)\rho+\gamma),\widehat{Z}^{\prime}_1((p-1)\rho+\sigma))=0.$$
Applying this repeatedly, we find that
$$\Hom_{G_1}(\widehat{Q}_1((p-1)\rho-\lambda),\widehat{Q}_1((p-1)\rho+w_0\mu))$$
is isomorphic as a $T$-module to
$$\bigoplus \Hom_{G_1}(\widehat{Z}_1((p-1)\rho+\gamma),\widehat{Z}^{\prime}_1((p-1)\rho+\sigma)).$$
Also,
$$\Hom_{G_1}(\widehat{Z}_1((p-1)\rho+\gamma),\widehat{Z}^{\prime}_1((p-1)\rho+\sigma)) \cong \begin{cases}  \sigma-\gamma & \textup{if } \sigma-\gamma \in p\mathbb{X} \\ 0 & \textup{otherwise}\\ \end{cases}.$$
From this it follows that
\begin{align*}
\textup{ch}\left(\bigoplus \Hom_{G_1}(\widehat{Z}_1((p-1)\rho+\gamma),\widehat{Z}^{\prime}_1((p-1)\rho+\sigma))\right)& = \pi_p(q(-w_0\lambda)^*q(\mu))\\
& = \pi_p(q(\lambda)q(\mu)).\\
\end{align*}
\end{proof}

\begin{cor}\label{C:tiltingchar}
For all $\lambda,\mu \in \mathbb{X}_1$, $\pi_p(q(\lambda)q(\mu))$ is a proper character, and $\pi_p(t(\lambda)t(\mu))$ is the character of a tilting module.
\end{cor}

\subsection{Examples}

We give a few examples of $q(\lambda)$ and $t(\lambda)$.

\begin{example}

Let $G=SL_2$.  In this case, the $G_1$-PIMs lift to tilting modules for all $p$, so $q(\lambda)=t(\lambda)$.  The Weyl group $W \cong S_2$, and $\mathbb{X} = \mathbb{Z}\varpi_1$.  We write $a\varpi_1$ simply as $a$, so we identify $\mathbb{X}_1$ with the set $\{0,1,\ldots,p-1\}$.  For $n \in \mathbb{X}_1$, there is no dominant weight $m<n$ such that $m-1 \uparrow n-1$.  Thus we have that $b^n_n=1$ and $b_m^n=0$ if $m \ne n$.  Therefore
$$t(n) = s(n), \quad 0 \le n \le p-1.$$ 
\end{example}

\begin{example}
Let $G=SL_3$.  Again, the $G_1$-PIMs always lift to tilting modules.  We have $\mathbb{X} = \mathbb{Z}\varpi_1+\mathbb{Z}\varpi_2$, and we write $(a,b)$ for $a\varpi_1+b\varpi_2$.  For $0 \le a,b \le p-1$, the following hold:

$$t(a,b) = \begin{cases} s(a,b) & \text{if } a+b \le p\\
s(a,b)+s(p-b,p-a) & \text{if } a+b >p\\
\end{cases}$$
\end{example}

\begin{example}
Let $G$ be the simple group having root system $\textup{G}_2$, and let $p=2$.  We have $\mathbb{X} = \mathbb{Z}\varpi_1+\mathbb{Z}\varpi_2$, and again we write $(a,b)$ for $a\varpi_1+b\varpi_2$.  Using the computations in \cite[Chapter 18]{Hu} one can show that
\begin{align*}
q(0,0) &= e(0,0)\\
q(1,0) & = s(1,0)\\
q(0,1) & = s(0,1)+s(1,0)\\
q(1,1) & = s(1,1)+2s(0,1)+2s(1,0).\\
\end{align*}
The results in \cite{BNPS2} show that $t(1,1) \ne q(1,1)$.  Using the results in this paper, one can prove this in another way, and in so doing obtain an alternate proof that the tilting module conjecture fails in this setting.  Using Corollary \ref{C:tiltingchar}, one can argue that the multiplicity of $s(1,0)$ in $t(1,1)$ must be at least $3$, for otherwise the character of
$$\Hom_{G_1}(T(2,2),T(2,2))^{(-1)}$$
will not be large enough to be the character of a tilting module (this method, however, does not settle whether the multiplicity is $3$ or $4$, a question to which we still do not know the answer).  We leave the details to the interested reader.
\end{example}

\section{Future Directions}

Fix a root system $\Phi$, and let $G$ be the simply connected simple group over $\Bbbk$ having $\Phi$ as its root system.  Our next goal will be to study properties of the characters $t(\lambda)$ as the characteristic of $\Bbbk$ gets larger, and in particular once $p \ge h$.  We know that it will hold that $t(\lambda)=q(\lambda)$ once $p = 2h-2$, and possibly sooner.   At some point Lusztig's Conjecture will also hold (a precise bound is still unknown) and the orbit multiplicities will stabilize for all higher $p$.

What happens below the range in which Lusztig's Conjecture holds is the content of \cite{RW2}, where it is shown that the patterns will be controlled by the $p$-Kazhdan-Lusztig polynomials.  Thus, studying the orbit multiplicities in $t(\lambda)$ is just another way of formulating the investigation in \cite{RW2}.  At the same time, our hope is that by keeping track of the character $t(\lambda)$, certain global patterns will become more clear.  We view the monotonicity property of the orbit multiplicities in individual $t(\lambda)$ as evidence for this hope.

\providecommand{\bysame}{\leavevmode\hbox
to3em{\hrulefill}\thinspace}

\end{document}